\documentclass{amsart}
\usepackage{amssymb,amstext,amsmath,amscd,amsthm,amsfonts,bm,enumerate,graphicx,latexsym,stmaryrd,multicol}

\usepackage[usenames]{color}
\usepackage[all]{xy}
\newtheorem{thm}{Theorem}[section]
\newtheorem{lem}[thm]{Lemma}

\newtheorem{cor}[thm]{Corollary}
\theoremstyle{definition}

\newtheorem{rem}[thm]{Remark}

\newtheorem*{claim*}{Claim}
\theoremstyle{remark}

\newtheorem*{ac}{Acknowlegments}

\numberwithin{equation}{thm}
\begin{document}
\allowdisplaybreaks
\title{Criteria for finite injective dimension of modules over a local ring}
\author{Shinnosuke Kosaka}
\address[Kosaka]{Graduate School of Mathematics, Nagoya University, Furocho, Chikusaku, Nagoya 464-8602, Japan}
\email{kosaka.shinnosuke.d2@s.mail.nagoya-u.ac.jp}
\author{Yuki Mifune}
\address[Mifune]{Graduate School of Mathematics, Nagoya University, Furocho, Chikusaku, Nagoya 464-8602, Japan}
\email{yuki.mifune.c9@math.nagoya-u.ac.jp}
\author{Kenta Shimizu}
\address[Shimizu]{Graduate School of Mathematics, Nagoya University, Furocho, Chikusaku, Nagoya 464-8602, Japan}
\email{kenta.shimizu.c1@math.nagoya-u.ac.jp}
\begin{abstract}
Let $R$ be a commutative Noetherian local ring. We prove that the finiteness of the injective dimension of a finitely generated $R$-module $C$ is determined by the existence of a Cohen--Macaulay module $M$ that satisfies an inequality concerning multiplicity and type, together with the vanishing of finitely many Ext modules. As applications, we recover a result of Rahmani and Taherizadeh and provide sufficient conditions for a finitely generated $R$-module to have finite injective dimension.
\end{abstract}

\maketitle
\section{Introduction}
Let $R$ be a commutative Noetherian local ring.
One of the most celebrated results characterizing the Cohen--Macaulay property of $R$ through the existence of special $R$-modules is the Bass conjecture, which was proved affirmatively using the Peskine--Szpiro intersection theorem (see \cite[Corollary 9.6.2, Remarks 9.6.4]{BH} for details).
Furthermore, Roberts \cite{R} characterized the Cohen--Macaulay property of $R$ by the existence of a Cohen--Macaulay module with finite projective dimension. Following this, Takahashi \cite{T2004} characterized the Gorenstein property of $R$ in terms of $\operatorname{G}$-dimension.
%%%%%%%%%%%%%%%%%%%%
The following result is part of \cite[Theorem 3.4]{RT}, which provides a criterion for a semidualizing $R$-module $C$ to be dualizing.
\begin{thm}[Rahmani and Taherizadeh]\label{RT_int}
Let $R$ be a Noetherian local ring and $C$ a semidualizing $R$-module. Assume that $C$ has type one and there exists a Cohen--Macaulay $R$-module $M$ of finite $\operatorname{G}_{C}$-dimension. Then $R$ is Cohen--Macaulay and $C$ is a canonical module of $R$.
\end{thm}
%%%%%%%%%%%%%%%%%%%%
Here, a finitely generated $R$-module $C$ is called {\em semidualizing} if the natural homomorphism $R\to\operatorname{End}_{R}(C)$ is an isomorphism and one has $\operatorname{Ext}_{R}^{i}(C,C)=0$ for all $i>0$.
A finitely generated $R$-module $X$ is said to be {\em totally $C$-reflexive} if the natural homomorphism $X\to\operatorname{Hom}_{R}(\operatorname{Hom}_{R}(X,C),C)$ is an isomorphism and one has $\operatorname{Ext}_{R}^{i}(X,C)=\operatorname{Ext}_{R}^{i}(\operatorname{Hom}_{R}(X,C),C)=0$ for all $i>0$.
The {\em $\operatorname{G}_{C}$-dimension} of an $R$-module $M$ is defined as the infimum of the lengths of resolutions of $M$ by totally $C$-reflexive modules.
The case $C=R$ in Theorem \ref{RT_int} was proved by Takahashi \cite{T2004}, and it gives a characterization of Gorenstein rings.
%%%%%%%%%%%%%%%%%%%%
The main result of this paper is the following theorem.
\begin{thm}[Theorem \ref{main thm}]\label{main thm int}
Let $R$ be a Noetherian local ring, and $C$ a finitely generated $R$-module with $\operatorname{depth}_{R}C=r$.
Assume that there exists a Cohen--Macaulay $R$-module $M$ of dimension $s$ satisfying the following two conditions:
\begin{enumerate}[\rm(a)]
\item
The inequality $\operatorname{r}_{R}(C)\operatorname{e}(M)\leq \operatorname{e}(\operatorname{Ext}_{R}^{r-s}(M,C))$ holds.
\item
One has $\operatorname{Ext}_{R}^{i}(M,C)=0$ for all $r-s+1\leq i \leq r+1$.
\end{enumerate}
Then $R$ is Cohen--Macaulay and $C$ is a maximal Cohen--Macaulay $R$-module with finite injective dimension.
Moreover, every Cohen--Macaulay $R$-module $N$ of dimension $s$ satisfies the conditions obtained by replacing $M$ with $N$ in the two conditions above.
\end{thm}
%%%%%%%%%%%%%%%%%%%%
Here, $\operatorname{r}_{R}(C)$ is the type of $C$, and $\operatorname{e}(M)$ is the multiplicity of $M$ with respect to the maximal ideal of $R$.
Note that if $R$ is a homomorphic image of a Gorenstein ring, then the conclusion says that $C$ is a finite direct sum of copies of a canonical module of $R$.
A notable improvement in the above theorem is that only finitely many vanishing conditions of Ext modules are required to determine the finiteness of the injective dimension of $C$. This is in contrast to the case where the $R$-module $M$ has finite $\operatorname{G}_{C}$-dimension.
Theorem \ref{main thm int} not only recovers Theorem \ref{RT_int} (see Remark \ref{recover_RT}) but also provides several other applications.
The following corollary gives sufficient conditions for a finitely generated $R$-module to have finite injective dimension.
%%%%%%%%%%%%%%%%%%%%
\begin{cor}[Corollaries \ref{cor_Gor}, \ref{cor:case_M_is_R}, \ref{cor:rank}, and \ref{cor_M=C}]\label{cor_int}
Let $R$ be a Noetherian local ring and $C$ a finitely generated $R$-module. Then the following hold.
\begin{enumerate}[\rm(1)]
\item 
Suppose that there exists a Cohen--Macaulay $R$-module $M$ satisfying the following three conditions:
\begin{enumerate}[\rm(i)]
\item
The dimension of $M$ is equal to $\operatorname{depth}R$.
\item
The inequality $\operatorname{r}(R)\operatorname{e}(M)\leq \operatorname{e}(\operatorname{Hom}_{R}(M,R))$ holds.
\item
One has $\operatorname{Ext}_{R}^{i}(M,R)=0$ for all $1\leq i \leq \operatorname{depth}R+1$.
\end{enumerate}
Then $R$ is Gorenstein.
\item
Suppose that $C$ is a Cohen--Macaulay $R$-module of dimension $n\geq0$ satisfying the following two conditions:
\begin{enumerate}[\rm(i)]
\item
The inequality $\operatorname{r}_{R}(C)\operatorname{e}(C)\leq \operatorname{e}(\operatorname{End}_{R}(C))$ holds.
\item
One has $\operatorname{Ext}_{R}^{i}(C,C)=0$ for all $1\leq i \leq n+1$.
\end{enumerate}
Then $R$ is Cohen--Macaulay and $C$ is a maximal Cohen--Macaulay $R$-module with finite injective dimension.
\item
Assume that $R$ is Cohen--Macaulay and $C$ is a maximal Cohen--Macaulay $R$-module. If either of the following two conditions holds, then $C$ has finite injective dimension.
\begin{enumerate}[\rm(i)]
\item
The inequality $\operatorname{r}_{R}(C)\operatorname{e}(R)\leq \operatorname{e}(C)$ holds.
\item
The $R$-module $C$ has a rank and the inequality $\operatorname{r}_{R}(C)\leq \operatorname{rank}_{R}C$ holds.
\end{enumerate}
\end{enumerate}
\end{cor}
%%%%%%%%%%%%%%%%%%%%
%%%%%%%%%%%%%%%%%%%%
\section{Proof of main theorem}
%%%%%%%%%%%%%%%%%%%%
In this section, we prove the main theorem stated in the Introduction and present a couple of corollaries. We begin with a fundamental lemma concerning the multiplicity and the length of a module.
For an $R$-module $M$ of finite length, we denote by $\ell(M)$ the length of $M$.
\begin{lem}
Let $(R,\mathfrak{m})$ be a Noetherian local ring and $M$ a Cohen--Macaulay $R$-module. Assume that there exists a sequence $\bm{x}$ of elements in $R$ which forms a system of parameters of $M$ and generates a reduction of $\mathfrak{m}$ with respect to $M$. Then the equality $\operatorname{e}(M)=\ell(M/\bm{x}M)$ holds.
\end{lem}
\begin{proof}
By \cite[Theorems 14.11 and 14.13]{Mat2}, we have $\operatorname{e}(M)=\operatorname{e}(\bm{x},M)=\operatorname{e}(\bm{x}+\operatorname{ann}_{R}M/\operatorname{ann}_{R}M,M)=\operatorname{e}(0,M/\bm{x}M)=\ell(M/\bm{x}M)$.
\end{proof}
%%%%%%%%%%%%%%%%%%%%
Next, we investigate the connection between regular sequences and the vanishing of Ext modules of a Cohen--Macaulay module.
\begin{lem}\label{lem:regular_M_to_Ext}
Let $R$ be a Noetherian local ring, and $C$ a finitely generated $R$-module with depth $r$.
Let $M$ be a Cohen--Macaulay $R$-module with dimension $s$.
Assume that one has $\operatorname{Ext}_{R}^{i}(M,C)=0$ for all $r-s+1\leq i \leq r+1$. 
Then every $M$-sequence $\bm{x}=x_{1},\ldots,x_{s}$  satisfies the following three conditions:
\begin{enumerate}[\rm(i)]
\item
The sequence $\bm{x}$ is an $\operatorname{Ext}_{R}^{r-s}(M,C)$-sequence.
\item
One has $\operatorname{Ext}_{R}^{r}(M/\bm{x}M,C)\cong \operatorname{Ext}_{R}^{r-s}(M,C)\otimes_{R}R/(\bm{x})$.
\item
One has $\operatorname{Ext}_{R}^{r+1}(M/\bm{x}M,C)=0$.
\end{enumerate}
\end{lem}
\begin{proof}
By induction on $s\geq0$. When $s=0$, the module $\operatorname{Ext}_{R}^{r}(M,C)$ is nonzero by \cite[Theorem 16.6]{Mat2}. Thus, the assertions hold.
Suppose that $s>0$. 
By virtue of \cite[Theorem 17.1]{Mat2}, one has $\operatorname{Ext}_{R}^{i}(M/x_{1}M,C)=0$ for all $i<r-s+1$. 
Combining the assumption with the exact sequence
$0\to M\xrightarrow{x_{1}}M\to M/x_{1}M\to 0$, 
we obtain the following exact sequence
\[
0\to \operatorname{Ext}_{R}^{r-s}(M,C)\xrightarrow{x_{1}} \operatorname{Ext}_{R}^{r-s}(M,C) \to \operatorname{Ext}_{R}^{r-s+1}(M/x_{1}M,C)\to 0,
\]
and $\operatorname{Ext}_{R}^{i}(M/x_{1}M,C)=0$ for all $r-(s-1)+1\leq i \leq r+1$. Hence, $x_{1}$ is an $\operatorname{Ext}_{R}^{r-s}(M,C)$-regular element and one has $\operatorname{Ext}_{R}^{r-s+1}(M/x_{1}M,C)\cong \operatorname{Ext}_{R}^{r-s}(M,C)\otimes_{R}R/(x_{1})$.
Applying the induction hypothesis to $M/x_{1}M$, we have the following:
\begin{enumerate}
\item
The sequence $x_{2},\ldots,x_{s}$ is an $\operatorname{Ext}_{R}^{r-s+1}(M/x_{1}M,C)$-sequence.
\item
One has $\operatorname{Ext}_{R}^{r}(M/\bm{x}M,C)\cong \operatorname{Ext}_{R}^{r-s+1}(M/x_{1}M,C)\otimes_{R}R/(x_{2},\ldots,x_{s})$.
\item
One has $\operatorname{Ext}_{R}^{r+1}(M/\bm{x}M,C)=0$.
\end{enumerate}
Combining the isomorphism $\operatorname{Ext}_{R}^{r-s+1}(M/x_{1}M,C)\cong \operatorname{Ext}_{R}^{r-s}(M,C)\otimes_{R}R/(x_{1})$ with (2) above, we obtain the isomorphism $\operatorname{Ext}_{R}^{r}(M/\bm{x}M,C)\cong \operatorname{Ext}_{R}^{r-s}(M,C)\otimes_{R}R/(\bm{x})$. This yields that $\operatorname{Ext}_{R}^{r-s}(M,C)$ is nonzero. Combining this with the fact that $x_{1}$ is an $\operatorname{Ext}_{R}^{r-s}(M,C)$-regular element and (1) above, the sequence $\bm{x}$ is an $\operatorname{Ext}_{R}^{r-s}(M,C)$-sequence.
\end{proof}
%%%%%%%%%%%%%%%%%%%%
Let $(R,\mathfrak{m},k)$ be a Noetherian local ring. For a nonzero finitely generated $R$-module $M$, we denote by $\operatorname{r}_{R}(M)$ the type of $M$, that is, $\operatorname{r}_{R}(M)=\operatorname{dim}_{k}\operatorname{Ext}_{R}^{r}(k,M)$, where $r=\operatorname{depth}_{R}M$.
The following lemma plays a crucial role in the proof of our main theorem.
\begin{lem}\label{lem:case_M_is_finite_length}
Let $(R,\mathfrak{m},k)$ be a Noetherian local ring, and $C$ a finitely generated $R$-module with $\operatorname{depth}_{R}C=r$.
Let $M$ be a nonzero $R$-module of finite length.
Assume that the following two conditions hold.
\begin{enumerate}[\rm(a)]
\item
The inequality $\operatorname{r}_{R}(C)\ell(M)\leq \ell(\operatorname{Ext}_{R}^{r}(M,C))$ holds.
\item
One has $\operatorname{Ext}_{R}^{r+1}(M,C)=0$.
\end{enumerate}
Then we have $\operatorname{Ext}_{R}^{r+1}(k,C)=0$. 
\end{lem}
\begin{proof}
We set $t=\operatorname{r}_{R}(C)$, and $l=\ell(M)$.
Take a composition series $0=N_{l}\subsetneq \cdots N_{1}\subsetneq N_{0}=M$ of $M$.
Then there exists an exact sequence 
\[
0\to N_{i}\to N_{i-1}\to k\to 0
\]
for each $1\leq i\leq l$.
Note that $\operatorname{Ext}_{R}^{j}(L,C)=0$ for all $R$-modules $L$ of finite length and integers $j<r$ by \cite[Theorem 17.1]{Mat2}.
Applying the functor $\operatorname{Hom}_{R}(-,C)$ to the above exact sequence, we obtain the following exact sequence for each $1\leq i\leq l$.
\[
0 \to k^{\oplus t}\to \operatorname{Ext}_{R}^{r}(N_{i-1},C)\to \operatorname{Ext}_{R}^{r}(N_{i},C).
\]
Hence, we have $\ell(\operatorname{Ext}_{R}^{r}(N_{i-1},C))\leq \ell(\operatorname{Ext}_{R}^{r}(N_{i},C))+t$ for each $1\leq i \leq r$. 
Therefore, we obtain the following:
\begin{align*}
        \ell(\operatorname{Ext}_{R}^{r}(M,C))&\leq \ell(\operatorname{Ext}_{R}^{r}(N_{1},C))+t\\
                                    &\leq \ell(\operatorname{Ext}_{R}^{r}(N_{2},C))+2t\\
                                    &\leq \cdots\\
                                    &\leq \ell(\operatorname{Ext}_{R}^{r}(N_{l},C))+lt\\
                                    &=tl.
\end{align*}
Combining this with the assumption, the equality $\ell(\operatorname{Ext}_{R}^{r}(M,C))=tl$ holds and all the above inequalities turn into equalities. In particular, one has $\ell(\operatorname{Ext}_{R}^{r}(M,C))= \ell(\operatorname{Ext}_{R}^{r}(N_{1},C))+t$.
The assumption and the exact sequence $0\to N_{1}\to M\to k\to 0$ yield an exact sequence
\[
0\to k^{\oplus t} \to \operatorname{Ext}_{R}^{r}(M,C)\to \operatorname{Ext}_{R}^{r}(N_{1},C)\to \operatorname{Ext}_{R}^{r+1}(k,C) \to 0.
\]
This implies that the equality $\ell(\operatorname{Ext}_{R}^{r+1}(k,C))+\ell(\operatorname{Ext}_{R}^{r}(M,C))= \ell(\operatorname{Ext}_{R}^{r}(N_{1},C))+t$ holds.
Thus, we have $\ell(\operatorname{Ext}_{R}^{r+1}(k,C))=0$, and the assertion follows.
\end{proof}
%%%%%%%%%%%%%%%%%%%%
Now we have reached the main result of this section.
For an $R$-module $M$, we denote by $\operatorname{E}_{R}(M)$ the injective hull of $M$.
\begin{thm}\label{main thm}
Let $(R,\mathfrak{m},k)$ be a Noetherian local ring, and $C$ a finitely generated $R$-module with $\operatorname{depth}_{R}C=r$.
Assume that there exists a Cohen--Macaulay $R$-module $M$ of dimension $s$ satisfying the following two conditions:
\begin{enumerate}[\rm(a)]
\item
The inequality $\operatorname{r}_{R}(C)\operatorname{e}(M)\leq \operatorname{e}(\operatorname{Ext}_{R}^{r-s}(M,C))$ holds.
\item
One has $\operatorname{Ext}_{R}^{i}(M,C)=0$ for all $r-s+1\leq i \leq r+1$.
\end{enumerate}
Then $R$ is Cohen--Macaulay and $C$ is a maximal Cohen--Macaulay $R$-module with finite injective dimension.
Moreover, every Cohen--Macaulay $R$-module $N$ of dimension $s$ satisfies the conditions obtained by replacing $M$ with $N$ in the two conditions above.
\end{thm}
\begin{proof}
We may assume that $k$ is infinite; see \cite[Lemma 8.4.2]{HS} for instance. 
Combining \cite[Theorem 1.1]{FFGR} and \cite[Corollary 9.6.2, Remarks 9.6.4, and Theorem 3.1.17]{BH}, it suffices to show that $\operatorname{Ext}_{R}^{r+1}(k,C)=0$ for the assertion in the first half.  
By virtue of \cite[Theorem 14.14.]{Mat2}, we can take a sequence $\bm{y}=\overline{x_{1}},\ldots, \overline{x_{s}}$ of elements in $R/\operatorname{ann}_{R}M$ such that $\bm{y}$ forms a system of parameters of $R/\operatorname{ann}_{R}M$ and the ideal $(\bm{y})$ is a reduction of $\mathfrak{m}/\operatorname{ann}_{R}M$.
Note that the ideal generated by $\bm{x}=x_{1},\ldots,x_{s}$ is a reduction of $\mathfrak{m}$ with respect to both $M$ and $\operatorname{Ext}_{R}^{r-s}(M,C)$.
Combining the fact that $M$ is Cohen--Macaulay and Lemma \ref{lem:regular_M_to_Ext}, we see that $\bm{x}$ is both an $M$-sequence and an $\operatorname{Ext}_{R}^{r-s}(M,C)$-sequence.
Hence, we have the following equalities:
\begin{align*}
\operatorname{e}(M)&=\operatorname{e}(\bm{x},M)=\ell(M/\bm{x}M), \text{ and} \\
\operatorname{e}(\operatorname{Ext}_{R}^{r-s}(M,C))&=\operatorname{e}(\bm{x},\operatorname{Ext}_{R}^{r-s}(M,C))=\ell(\operatorname{Ext}_{R}^{r-s}(M/\bm{x}M,C)).
\end{align*}
By the assumption and Lemma \ref{lem:regular_M_to_Ext}, the inequality $\operatorname{r}_{R}(C)\ell(M/\bm{x}M)\leq \ell(\operatorname{Ext}_{R}^{r-s}(M/\bm{x}M,C))$ holds and one has $\operatorname{Ext}_{R}^{r+1}(M/\bm{x}M,C)=0$.
Therefore, we conclude that $\operatorname{Ext}_{R}^{r+1}(k,C)=0$ by Lemma \ref{lem:case_M_is_finite_length}. 

We shall prove the assertion in the latter part. 
We may assume that $R$ is complete. It follows that $C\cong\omega_{R}^{\oplus t}$, where $t=\operatorname{r}_{R}(C)$ and $\omega_{R}$ is a canonical module of $R$; see \cite[Exercises 3.3.28(a)]{BH} for instance.
Let $N$ be a Cohen--Macaulay $R$-module of dimension $s$. By \cite[Theorems 3.5.7 and 3.5.8]{BH}, one has $\operatorname{Ext}_{R}^{i}(N,C)\cong \operatorname{Hom}_{R}(H_{\mathfrak{m}}^{d-i}(N),\operatorname{E}_{R}(k))^{\oplus t}=0$ for all integers $i\neq d-s$. Therefore, the condition (b) for $N$ holds. Note that the condition (a) for $N$ holds if and only if the inequality $\operatorname{e}(N)\leq \operatorname{e}(\operatorname{Ext}_{R}^{d-s}(N,\omega_{R}))$ holds.
Since $\operatorname{grade}_{R}N=d-s$, we can take an $R$-sequence $\bm{x}=x_{1},\ldots, x_{d-s}$ in $\operatorname{ann}_{R}N$. It follows that $N$ is a maximal Cohen--Macaulay $R/(\bm{x})$-module, and $\bm{x}$ is an $\omega_{R}$-sequence.
Hence, we have $\operatorname{Ext}_{R}^{d-s}(N,\omega_{R})\cong \operatorname{Hom}_{R/(\bm{x})}(N,\omega_{R}/\bm{x}\omega_{R})\cong\operatorname{Hom}_{R/(\bm{x})}(N,\omega_{R/(\bm{x})})$.
Thus, it suffices to prove the following.
\begin{claim*}
Let $(R,\mathfrak{m},k)$ be a Cohen--Macaulay local ring with a canonical module $\omega_{R}$. If $M$ is a maximal Cohen--Macaulay $R$-module, then one has $\operatorname{e}(M)=\operatorname{e}(\operatorname{Hom}_{R}(M,\omega_{R}))$.  
\end{claim*}
We may assume that $k$ is infinite. 
By \cite[Theorem 14.14]{Mat2}, there exists a sequence $\bm{x}=x_{1},\ldots,x_{d}$ of elements in $R$ such that $\bm{x}$ forms a system of parameters of $R$, and $(\bm{x})$ is a reduction of $\mathfrak{m}$.
As in the proof of the first part, applying \cite[Propositions 3.3.3(a) and 3.2.12]{BH} yields $\operatorname{e}(M)=\ell(M/\bm{x}M)=\ell(\operatorname{Hom}_{R/(\bm{x})}(M/\bm{x}M,\omega_{R}/\bm{x}\omega_{R}))=\ell(\operatorname{Hom}_{R}(M,\omega_{R})\otimes_{R}R/(\bm{x}))=\operatorname{e}(\operatorname{Hom}_{R}(M,\omega_{R}))$.
This completes the proof.
\end{proof}
%%%%%%%%%%%%%%%%%%%%
\begin{rem}\label{recover_RT}
The result stated as Theorem \ref{RT_int} by Rahmani and Taherizadeh can be recovered using Theorem \ref{main thm} as follows. 
Factoring $M$ by a maximal $M$-sequence, we may assume that $M$ has finite length (for the basic properties of $\operatorname{G}_{C}$-dimension, see \cite{G}).
We set $r=\operatorname{depth}R$ and $N=\operatorname{Ext}_{R}^{r}(M,C)$.
Note that $N$ has finite length, and $\operatorname{G}_{C}\operatorname{dim}_{R}M=r$.
Take a sequence $\bm{x}=x_{1},\ldots,x_{r}$ of elements in $\operatorname{ann}_{R}M$ that forms both an $R$-sequence and a $C$-sequence.
Then it follows that $M$ is a totally $C/\bm{x}C$-reflexive $R/(\bm{x})$-module and so is $\operatorname{Hom}_{R/(\bm{x})}(M,C/\bm{x}C)\cong N$.
Hence, we have $\operatorname{Ext}_{R}^{r+1}(M,C)\cong\operatorname{Ext}^{1}_{R/(\bm{x})}(M,C/\bm{x}C)=0$, and  $\operatorname{Ext}_{R}^{r+1}(N,C)\cong\operatorname{Ext}^{1}_{R/(\bm{x})}(N,C/\bm{x}C)=0$.
Moreover, one has $\operatorname{Ext}_{R}^{r}(N,C)\cong\operatorname{Hom}_{R/(\bm{x})}(\operatorname{Hom}_{R/(\bm{x})}(M,C/\bm{x}C),C/\bm{x}C)\cong M$.
Combining these with Theorem \ref{main thm} and \cite[Proposition 3.3.13]{BH}, the assertion holds.
\end{rem}
%%%%%%%%%%%%%%%%%%%%
As an application of Theorem \ref{main thm}, we obtain the following corollary. This result characterizes the Gorenstein property of $R$ in terms of the existence of a Cohen--Macaulay module that satisfies certain conditions, concerning specific numerical invariants and the vanishing of finitely many Ext modules.
\begin{cor}\label{cor_Gor}
Let $R$ be a Noetherian local ring with $\operatorname{depth}R=r$.
Then the following conditions are equivalent:
\begin{enumerate}[\rm(i)]
\item
$R$ is Gorenstein.
\item
There exists a Cohen--Macaulay $R$-module $M$ satisfying the following three conditions:
\begin{enumerate}[\rm(1)]
\item
The dimension of $M$ is equal to $r$.
\item
The inequality $\operatorname{r}(R)\operatorname{e}(M)\leq \operatorname{e}(\operatorname{Hom}_{R}(M,R))$ holds.
\item
One has $\operatorname{Ext}_{R}^{i}(M,R)=0$ for all $1\leq i \leq r+1$.
\end{enumerate}
\end{enumerate}
\end{cor}
\begin{proof}
(i)$\Rightarrow$(ii): It suffices to take $M=R$.
(ii)$\Rightarrow$(i): Taking $C=R$ in Theorem \ref{main thm}, we conclude that $R$ has finite injective dimension.
\end{proof}
%%%%%%%%%%%%%%%%%%%%
The following result shows that the finiteness of the injective dimension of a maximal Cohen--Macaulay module is determined by a condition on numerical invariants.
\begin{cor}\label{cor:case_M_is_R}
Let $(R,\mathfrak{m},k)$ be a Cohen--Macaulay local ring and $C$ a maximal Cohen--Macaulay $R$-module. If the inequality $\operatorname{r}_{R}(C)\operatorname{e}(R)\leq \operatorname{e}(C)$ holds, then $C$ has finite injective dimension.
\end{cor}
\begin{proof}
Taking $M=R$ in Theorem \ref{main thm}, the assertion holds.
\end{proof}
%%%%%%%%%%%%%%%%%%%%
\begin{rem}
There is a direct proof of Corollary \ref{cor:case_M_is_R} without using Theorem \ref{main thm}:
firstly, we consider the case where $d=\operatorname{dim}R=0$.
Note that $\operatorname{e}(R)=\ell(R)$, $\operatorname{e}(C)=\ell(C)$, and $\operatorname{E}_{R}(C)\cong\operatorname{E}_{R}(k)^{\oplus\operatorname{r}_{R}(C)}$. It follows that there exists an exact sequence of the following form: 
\[
0\to C\to\operatorname{E}_{R}(k)^{\oplus\operatorname{r}_{R}(C)}\to D\to0.
\]
Combining this with the assumption, one has $\operatorname{r}_{R}(C)\ell(R)\leq\ell(C)\leq\ell(C)+\ell(D)=\operatorname{r}_{R}(C)\ell(\operatorname{E}_{R}(k))=\operatorname{r}_{R}(C)\ell(R)$.
This implies that $\ell(D)=0$, and that $C\cong\operatorname{E}_{R}(k)^{\oplus\operatorname{r}_{R}(C)}$ has injective dimension zero.
Next, we consider the general case.
We may assume that $k$ is infinite. 
By \cite[Theorem 14.14]{Mat2}, we can take an $R$-sequence $\bm{x}=x_{1},\ldots,x_{d}$ such that $(\bm{x})$ is a reduction of $\mathfrak{m}$. 
Hence, the assumption implies that the inequality $\operatorname{r}_{R/(\bm{x})}(C/\bm{x}C)\ell(R/(\bm{x}))\leq \ell(C/\bm{x}C)$ holds.
Therefore, from the case of dimension zero, it follows that $\operatorname{id}_{R/(\bm{x})}(C/\bm{x}C)=0$, and this implies that $\operatorname{id}_{R}C=d<\infty$.
\end{rem}
%%%%%%%%%%%%%%%%%%%%
Let $R$ be a commutative Noetherian ring and $M$ a finitely generated $R$-module. For a nonnegative integer $r$, we say that $M$ has {\em rank $r$}, denoted by $\operatorname{rank}_{R}M=r$, if $M\otimes_{R}Q$ is a free $Q$-module of rank $r$, where $Q$ is the total ring of fractions of $R$.
Note that $M$ has rank $0$ if and only if the grade of $M$ is positive.
The following result gives a sufficient condition for a maximal Cohen--Macaulay module that has a rank to admit finite injective dimension.
\begin{cor}\label{cor:rank}
Let $R$ be a Cohen--Macaulay local ring and $C$ a maximal Cohen--Macaulay $R$-module. Assume that $C$ has a rank and the inequality $\operatorname{r}_{R}(C)\leq \operatorname{rank}_{R}C$ holds.
Then $C$ has finite injective dimension.
\end{cor}
\begin{proof}
Since $C$ is a maximal Cohen--Macaulay $R$-module, it has positive rank. By \cite[Corollary 4.7.9]{BH}, it follows that $\operatorname{e}(C)=\operatorname{e}(R)\operatorname{rank}_{R}(C)$.
Hence, we have $\operatorname{r}_{R}(C)\operatorname{e}(R)\leq \operatorname{rank}_{R}(C)\operatorname{e}(R)=\operatorname{e}(C)$. Thus by Corollary \ref{cor:case_M_is_R}, we obtain $\operatorname{id}_{R}C<\infty$.
\end{proof}
%%%%%%%%%%%%%%%%%%%%
\begin{rem}
Let $R$ be a Cohen--Macaulay local ring and $C$ a maximal Cohen--Macaulay $R$-module with a (positive) rank. If $\operatorname{r}_{R}(C)=1$, then $C$ has finite injective dimension by \cite[Propositions 1.4.3 and 3.3.13]{BH}. Thus, Corollary \ref{cor:rank} is a generalization of the case where $C$ has type one.
\end{rem}
%%%%%%%%%%%%%%%%%%%%
We close this section by giving a result obtained by taking $M=C$ in Theorem \ref{main thm}. 
\begin{cor}\label{cor_M=C}
Let $R$ be a Noetherian local ring, and $C$ a Cohen--Macaulay $R$-module of dimension $n$ satisfying the following two conditions:
\begin{enumerate}[\rm(a)]
\item
The inequality $\operatorname{r}_{R}(C)\operatorname{e}(C)\leq \operatorname{e}(\operatorname{End}_{R}(C))$ holds.
\item
One has $\operatorname{Ext}_{R}^{i}(C,C)=0$ for all $1\leq i \leq n+1$.
\end{enumerate}
Then $R$ is Cohen--Macaulay and $C$ is a maximal Cohen--Macaulay $R$-module with finite injective dimension.
\end{cor}
\begin{rem}
Let $R$ be a Noetherian local ring.
For an integer $n\geq2$ and a finitely generated $R$-module $C$, we say that $C$ is {\em $n$-semidualizing} if the natural homomorphism $R\to\operatorname{End}_{R}(C)$ is an isomorphism and one has $\operatorname{Ext}_{R}^{i}(C,C)=0$ for all $1\leq i\leq n$.
The terminology is based on \cite{T2007}.
Let $C$ be a Cohen--Macaulay $R$-module of dimension $n>0$. If $C$ is $(n+1)$-semidualizing and the inequality $\operatorname{r}_{R}(C)\operatorname{e}(C)\leq \operatorname{e}(R)$ holds, then $R$ is Cohen--Macaulay and $C$ is a canonical module of $R$.
Indeed, it follows that $R$ is Cohen--Macaulay and $C$ is a maximal Cohen--Macaulay $R$-module with finite injective dimension by Corollary \ref{cor_M=C}.
Hence, we have $\widehat{R}\cong\operatorname{End}_{\widehat{R}}(\widehat{C})\cong\widehat{R}^{\oplus t^{2}}$, where $t=\operatorname{r}_{R}(C)$. This implies that $t=1$. Consequently, $C$ is a faithful maximal Cohen--Macaulay $R$-module of type one. Therefore, the assertion follows from \cite[Proposition 3.1.13]{BH}.
\end{rem}
%%%%%%%%%%%%%%%%%%%%
%%%%%%%%%%%%%%%%%%%%
\begin{ac}
The authors would like to thank their supervisor Ryo Takahashi for giving many thoughtful questions and helpful discussions. They also thank Kaito Kimura and Yuya Otake for their valuable comments.
\end{ac}
%%%%%%%%%%%%%%%%%%%%
%%%%%%%%%%%%%%%%%%%%


\begin{thebibliography}{99}
\bibitem{BH}
{\sc W. Bruns; J. Herzog}, Cohen--Macaulay rings, revised edition, Cambridge Studies in Advanced Mathematics {\bf 39}, {\it Cambridge University Press, Cambridge}, 1998.
\bibitem{FFGR}
{\sc R. Fossum; H. -B. Foxby; P. Griffith; I. Reiten}, Minimal injective resolutions with
applications to dualizing modules and Gorenstein modules, {\em Inst. Hautes \'{E}tudes Sci. Publ. Math.} No. 45 (1975), 193--215.
\bibitem{G}
{\em E. S. Golod}, $G$-dimension and generalized perfect ideals. Algebraic geometry and its applications. {\em Trudy Mat. Inst. Steklov.} 165 (1984), 62--66.
\bibitem{HS}
{\sc C. Huneke; I. Swanson}, Integral closure of ideals, rings and modules, London Mathematical Society Lecture Note Series {\bf 336}, {\it Cambridge University Press, Cambridge}, 2006.
\bibitem{Mat2}
{\sc H. Matsumura}, Commutative ring theory, Translated from the Japanese by M. Reid, Second edition, Cambridge Studies in Advanced Mathematics {\bf 8}, {\it Cambridge University Press, Cambridge}, 1989.
\bibitem{RT}
{\sc M. Rahmani; A. Taherizadeh}, Dual of bass numbers and dualizing modules, {\em Comm. Algebra} {\bf 45} (2017), no. 4, 1493--1508.
\bibitem{R}
{\sc P. Roberts}, Multiplicities and Chern classes in local algebra, Cambridge Tracts in Mathematics, 133, {\em Cambridge University Press, Cambridge}, 1998.
\bibitem{T2007}
{\sc R. Takahashi}, A new approximation theory which unifies spherical and Cohen-Macaulay approximations, {\em J. Pure Appl. Algebra} {\bf 208} (2007), no. 2, 617--634.
\bibitem{T2004}
{\sc R. Takahashi}, Some characterizations of Gorenstein local rings in terms of G-dimension, {\em Acta Math. Hungar.} {\bf 104} (2004), no. 4, 315--322.
\end{thebibliography}
\end{document}